\documentclass[12pt]{amsart}
\newcommand{\fk}{\mathfrak{k}}
\newcommand{\ft}{\mathfrak{t}}
\begin{document}

\newcommand{\nc}{\newcommand}

\newcommand{\colvec}[2]{\left  ( \begin{array}{cc} #1  \\
     #2  \end{array} \right ) }

\newcommand{\Tr}{\,{\rm Tr}\,}
\newcommand{\End}{\,{\rm End}\,}
\newcommand{\Hom}{\,{\rm Hom}\,}

\newcommand{\Ker}{ \,{\rm Ker} \,}

\newcommand{\bla}{\phantom{bbbbb}}
\newcommand{\onebl}{\phantom{a} }
\newcommand{\eqdef}{\;\: {\stackrel{ {\rm def} }{=} } \;\:}
\newcommand{\sign}{\: {\rm sign}\: }
\newcommand{\sgn}{ \:{\rm sgn}\:}
\newcommand{\half}{ {\frac{1}{2} } }
\newcommand{\vol}{ \,{\rm vol}\, }

%

\newcommand{\beq}{\begin{equation}}
\newcommand{\eeq}{\end{equation}}
\newcommand{\beqst}{\begin{equation*}}
\newcommand{\eeqst}{\end{equation*}}
\newcommand{\barr}{\begin{array}}
\newcommand{\earr}{\end{array}}
\newcommand{\beqar}{\begin{eqnarray}}
\newcommand{\eeqar}{\end{eqnarray}}
\newtheorem{theorem}{Theorem}[section]
\newtheorem{corollary}[theorem]{Corollary}
\newtheorem{lemma}[theorem]{Lemma}
\newtheorem{prop}[theorem]{Proposition}
\newtheorem{definition}[theorem]{Definition}
\newtheorem{remit}[theorem]{Remark}
\newtheorem{conjecture}[theorem]{Conjecture}
\newcommand{\matr}[4]{\left \lbrack \begin{array}{cc} #1 & #2 \\
     #3 & #4 \end{array} \right \rbrack}

\newenvironment{rem}{\begin{remit}\rm}{\end{remit}}
\newcommand{\quott}{/\!/}
\newcommand{\hkquott}{/\!/\!/\!/}

\newcommand{\aff}{{\mathbb A}}
\newcommand{\RR}{{\mathbb R}}
\newcommand{\CC}{{\mathbb C}}
\nc{\FF}{ {\mathbb F} } 
\newcommand{\ZZ}{{\mathbb Z}}
\newcommand{\PP}{ {\mathbb P} }
\newcommand{\QQ}{{\mathbb Q}}
\newcommand{\UU}{{\mathbb U}}

\newcommand{\cala}{{\mbox{$\mathcal A$}}}
\newcommand{\calb}{{\mbox{$\mathcal B$}}}
\newcommand{\calc}{{\mbox{$\mathcal C$}}}
\newcommand{\cald}{{\mbox{$\mathcal D$}}}
\newcommand{\cale}{{\mbox{$\mathcal E$}}}
\newcommand{\calf}{{\mbox{$\mathcal F$}}}
\newcommand{\calg}{{\mbox{$\mathcal G$}}}
\newcommand{\calh}{{\mbox{$\mathcal H$}}}
\newcommand{\cali}{{\mbox{$\mathcal I$}}}
\newcommand{\calj}{{\mbox{$\mathcal J$}}}
\newcommand{\calk}{{\mbox{$\mathcal K$}}}
\newcommand{\call}{{\mbox{$\mathcal L$}}}
\newcommand{\calm}{{\mbox{$\mathcal M$}}}
\newcommand{\caln}{{\mbox{$\mathcal N$}}}
\newcommand{\calo}{{\mbox{$\mathcal O$}}}
\newcommand{\calp}{{\mbox{$\mathcal P$}}}
\newcommand{\calq}{{\mbox{$\mathcal Q$}}}
\newcommand{\calr}{{\mbox{$\mathcal R$}}}
\newcommand{\cals}{{\mbox{$\mathcal S$}}}
\newcommand{\calt}{{\mbox{$\mathcal T$}}}
\newcommand{\calu}{{\mbox{$\mathcal U$}}}
\newcommand{\calv}{{\mbox{$\mathcal V$}}}
\newcommand{\calw}{{\mbox{$\mathcal W$}}}
\newcommand{\calx}{{\mbox{$\mathcal X$}}}
\newcommand{\caly}{{\mbox{$\mathcal Y$}}}
\newcommand{\calz}{{\mbox{$\mathcal Z$}}}

\newcommand{\ib}{   }


%

%

\def\a{\alpha}
\def\b{\beta}
\def\g{\gamma}
\def\d{\delta}
\def\e{\epsilon}
\def\z{\zeta}
\def\h{\eta}
\def\t{\theta}
\def\k{\kappa}
\def\l{\lambda}
\def\m{\mu}
\def\n{\nu}
\def\x{\xi}
\def\p{\pi}
\def\r{\rho}
\def\s{\sigma}
\def\ta{\tau}
\def\u{\upsilon}
\def\ph{\phi}
\def\c{\chi}
\def\ps{\psi}
\def\o{\omega}

\def\G{\Gamma}
\def\D{\Delta}
\def\T{\Theta}
\def\L{\Lambda}
\def\X{\Xi}
\def\P{\Pi}
\def\U{\Upsilon}
\def\Ph{\Phi}
\def\Ps{\Psi}
\def\O{\Omega}

\nc{\itwopi}{ { 2 \pi i} }


 \setlength{\textwidth}{6.7in}
\setlength{\textheight}{7.5truein}
\setlength{\evensidemargin}{0in}
\setlength{\oddsidemargin}{0in}
\setlength{\topmargin}{0truein}
\setlength{\parskip}{0.2\baselineskip}


\newcommand{\cy}{C^\infty}
\nc{\su}{ {\frak{su} } }
\nc{\liet}{ \mathfrak{t} } 
\nc{\lieg}{\mathfrak{g} } 
\nc{\Lie}{\rm Lie}

\title{Intersection cohomology of the universal imploded cross-section of SU(3)}
\author{Nan-Kuo Ho}
\address{Department of Mathematics \\
National Tsing-Hua University \\ Hsinchu \\ Taiwan}
\email{nankuo@math.nthu.edu.tw}

\author{Lisa  Jeffrey}
\address{Department of Mathematics \\
University of Toronto \\ Toronto, Ontario \\ Canada}
\email{jeffrey@math.toronto.edu}
\date{22 October 2014}
\begin{abstract}{We compute the intersection cohomology of the universal imploded
cross-section of SU(3), and show that it is different from the intersection
cohomology of a point.}
\end{abstract}
\maketitle

\renewcommand{\quott}{/\!/}

\section{Introduction: purpose of imploded cross-sections}

Let $K$ be a compact Lie group with maximal torus $T$, and
let $M$ be a Hamiltonian $K$-manifold.
The Guillemin-Sternberg symplectic cross-section theorem tells us that the $K$-orbits of $\Phi^{-1}((\ft^*_+)^\circ)$ (the preimage of the interior of the positive Weyl Chamber $(\ft^*_+)^\circ$ under the moment map $\Phi$)
 is dense in $M$ and is symplectic with Hamiltonian $T$-action on
$\Phi^{-1}((\ft^*_+)^\circ)$.
 One motivation in \cite{GJS} was to {\em complete} this space in a nice way that is also compatible with  symplectic reduction. For points in the interior of the positive Weyl chamber, the stabilizer is the maximal torus $T$. However, for points on the boundary of $\ft^*_+$, the stabilizers are bigger than $T$ and the preimage under the moment map is not a symplectic manifold. So the idea of the imploded cross-section is that for those faces,
we must   collapse further
(taking the quotient by the commutator subgroup of their stabilizer)
so  that the quotient becomes a symplectic stratified space
 and that it would have a Hamiltonian $T$-action on each stratum. One beautiful property is that the symplectic quotient of the resulting imploded space by the $T$-action is the same as the original symplectic quotient of the manifold $M$ by the $K$-action.

The universal example of an imploded cross-section is defined as
the imploded
cross-section $(T^*K)_{\rm impl}$ of the cotangent bundle of the Lie group $K$.

It is important to understand the invariants of imploded cross sections.
In the simplest case $K = SU(2)$ the universal imploded cross-section is
$\CC^2$, so its topology is trivial.  In this paper we study the universal example for
$SU(3)$. The universal imploded cross-section for $SU(3)$ contracts to a point,
but its topological invariants are not all trivial. We compute its intersection cohomology (which is not a homotopy
invariant), and find that this invariant
distinguishes it from the
intersection cohomology of a point. In this paper all intersection cohomology groups
are computed with middle perversity, as in \cite{KW} which is our main reference, and all the cohomology groups and intersection cohomology groups are with real coefficients.

\vspace{6mm}
\noindent {\it Acknowledgments.}
We would like to thank Reyer Sjamaar for many useful discussions.

\section{Universal example for $SU(3)$}

For more detailed accounts of symplectic implosion, please see \cite{GJS}.

Let $(M,\omega)$ be a Hamiltonian $K$-space with moment map $\Phi:M\to \fk^*$. Fix a maximal torus $T$ and a closed fundamental Weyl Chamber $\ft^*_+$. Denote $K_\lambda$ the centralizer of $\lambda\in\fk^*$ and $[\cdot,\cdot]$ the commutator subgroup. Introduce an equivalence relation in $\Phi^{-1}(\ft^*_+)$ as follows: $\forall m_1,~m_2\in \Phi^{-1}(\ft^*_+)$, we say $m_1\sim m_2$ if $\exists k\in [K_{\Phi(m_1)},K_{\Phi(m_1)}]$ such that $m_2=km_1$.

\begin{definition}
The imploded cross-section of $M$ is the quotient space $M_{{\rm impl}}:=\Phi^{-1}(\ft^*_+)/\sim$, equipped with the quotient topology.
\end{definition}

\begin{prop}\cite[Theorem 3.4]{GJS}
The space
$M_{{\rm impl}}$ is a stratified (in the sense of Sjamaar-Lerman) Hamiltonian $T$-space with moment map $\Phi_{{\rm impl}}:M_{{\rm impl}}\to \ft^*_+$ induced by $\Phi$. Moreover, $\forall \lambda\in \ft^*_+$, the canonical map $\Phi^{-1}(\lambda)\to \Phi^{-1}_{\rm impl}(\lambda)$ induces an isomorphism of symplectic quotients $$M\quott_{{\mathcal{O}_\lambda}}K\cong M_{{\rm impl}}\quott_{\lambda}T.$$
Here, if $K$ is a nonabelian group,  the symbol
$M\quott_{\mathcal{O}_\lambda} K$ denotes the symplectic quotient of
a Hamiltonian $K$-manifold $M$ at
the orbit ${\mathcal{O}_\lambda}$.
Likewise, if $T$ is a torus and $M$ is a Hamiltonian $T$-space,
the symbol $M\quott_\lambda T$ denotes the symplectic quotient
of $M$ at $\lambda$.

\end{prop}

This proposition introduced a way to abelianize Hamiltonian $K$-space for any compact connected Lie group $K$.
The following proposition is one of the reasons why the study of the universal imploded space $(T^*K)_{{\rm impl}}$ is important, and why \cite{GJS} called $(T^*K)_{{\rm impl}}$ the \emph{universal} imploded space.

\begin{prop}\cite[Theorem 4.9]{GJS}
Consider the imploded cross-section $(T^*K)_{{\rm impl}}$ of the cotangent bundle $T^*K$. Then there is an isomorphism between Hamiltonian $T$-spaces $$M_{\rm impl}\cong (M\times (T^*K)_{{\rm impl}})\quott_{0} K$$ where the quotient is taken with respect to the diagonal $K$-action.
\end{prop}

The imploded cross-section of the universal example for $SU(3)$
(\cite[example 6.16]{GJS})
 is a complex hypersurface determined
  by the quadratic equation $\sum_{k=1}^3 w_k z_k = 0$
 where $w_k, z_k$ are complex variables ($k = 1,2, 3$).
This space has an isolated singularity at $0$, because
it can be constructed from the preimage of $0$ under
the holomorphic map $f: (z, w) \mapsto \sum_k w_k z_k$ and
the derivatives are $$df(\partial/\partial z_k) = w_k,
\quad \mbox{and}\quad df(\partial/\partial w_k ) = z_k $$
so  $$df = (z_1,z_2, z_3, w_1, w_2, w_3),$$
 which
equals $0$ only if $z_j = w_j = 0 $ for all $j$, so  $(0, \cdots, 0)$ is the
only singular point.

\begin{remit}
For moduli spaces of flat connections on 2-manifolds the imploded
cross-sections were defined in \cite{HJS}. These are
imploded cross-sections for quasi-Hamiltonian spaces.
For a compact Lie group $K$,
the quasi-Hamiltonian analogue of the universal imploded space is the
imploded cross-section of the double $DK = K \times K$.
When $K = SU(n)$, the imploded cross-section of the
 double $DK$ has a stratum whose
closure is a smooth even-dimensional sphere.
(See \cite{HJS}.)
\end{remit}

\section{Intersection cohomology}

Our main reference for this section is book \cite{KW}. Intersection homology and cohomology are designed to understand the
topology of singular spaces. Since the imploded cross-section is almost always singular (often worse than orbifold singularities, see \cite[Section 6]{GJS}), it is natural to consider the intersection cohomology of the space. The definition for intersection cohomology is somewhat lengthy so we do not state it here. Instead, we list some of its important properties.

\begin{prop} (cf: \cite{KW})
\begin{enumerate}
\item If the space is smooth, then its intersection homology (resp. cohomology) is the same as its ordinary homology (resp. cohomology).
\item Intersection homology (resp. cohomology) is a homeomorphism invariant but not a homotopy invariant.
\item Poincar\'{e} duality holds (cf: \cite[Theorem 5.1.1]{KW}).
\item Excision and the Mayer-Vietoris sequence are valid
 (cf: \cite[Section 4.6]{KW}).
\end{enumerate}
\end{prop}

For our purposes, we will also state the result for computing the intersection cohomology of spaces with isolated singularities as follows.

Let $U$ be an open neighbourhood of real
dimension $2n$
in a space with an isolated singularity
at $x$. Then the intersection homology of $U$ is, (see \cite[p. 94, Chap. 6.3]{KW})

$$ IH_i (U) \cong \begin{cases} H_i(U \setminus \{x\})  &  i \le n -1 \cr
 \mbox{Im} \left(H_i(U \setminus \{x\})  \to H_i(U)\right) &  i = n  \cr
 H_i(U) &  i \ge n+1  \cr
\end{cases}
$$

In fact
(\cite{KW} p. 95, taking duals of equation (6.13))
 we have the following characterization of the intersection cohomology of $U$:
$$IH^i(U) \cong \begin{cases} H^i(Y) & i \le n-1 \cr
0 & i \ge n \cr
\end{cases}
$$
when
$U $ is homeomorphic to the cone $C(Y)$ for a compact Riemannian manifold
$Y$ of real dimension $2n-1$.

\section{Intersection cohomology of $SU(3)$ universal example}

The universal imploded cross-section that we want to consider is
 $$(T^*SU(3))_{{\rm impl}}= \{(z,w) \in \CC^3 \times \CC^3: z \cdot w =0\}.$$
As explained in Section 2, this space has an isolated singularity at $(0,0)$.
Let $$U= \{(z,w) \in \CC^3 \times \CC^3: z \cdot w =0,
      |z|^2 + |w|^2 <  1 + \epsilon \}$$ be an open neighborhood of the singularity, then $IH^j((T^*SU(3))_{{\rm impl}})=IH^j(U)$, since $(T^*SU(3))_{{\rm impl}}$ is homeomorphic to $U$.

Consider a compact Riemannian manifold $Y$ defined by
      \begin{equation*}
Y = \{(z,w) \in \CC^3 \times \CC^3: z \cdot w =0,
      |z|^2 + |w|^2 =  1  \}.
\end{equation*}
Then the space $U$ is homeomorphic to the cone $C(Y)$, and $dim_\RR Y = 9$. According to Section 3, this means $n=5$ and $IH^j(U)$ is the same as $H^j(Y)$ for $0 \le j \le 4$, and
is $0$ otherwise.

Consider the subspace of real dimension $9$
$$ X = \{(z,w) \in (\CC^3 \setminus \{0\} )\times (\CC^3 \setminus \{0\}): z \cdot w = 0, |z|^2 + |w|^2  = 1   \}\subset Y.$$

The following Lemma gives us a deformation retraction from $X$ to $SU(3)$.
\begin{lemma} There is a surjective map
$$\{ (z, w) \in (\CC^3 \setminus \{0\}) \times (\CC^3 \setminus \{0\}):
z\cdot w = 0 \} \to SU(3)$$ defined as follows:

$$ (z,w) \mapsto M=\begin{pmatrix} z_1/|z| & \bar{w}_1/|w| & u_1/|u| \cr
z_2/|z| & \bar{w_2}/|w| & u_2/|u| \cr
z_3/|z| & \bar{w}_3/|w| & u_3/|u|\cr
\end{pmatrix} $$
where $u:= (u_1, u_2,u_3) \in \CC^3 \setminus \{0\}$
satisfy $\bar{z} \cdot u= 0 $, $w \cdot u = 0 $
and the determinant of $M$ is $1$.
\end{lemma}
\begin{proof}
We solve for $(u_1,u_2, u_3)$ as follows.
Let
\beq \label{eq1}
u_1 \bar{z}_1 + u_2 \bar{z}_2+u_3 \bar{z}_3 = 0 \eeq
\beq \label{eq2}
u_1 w_1 + u_2 w_2+u_3 w_3 = 0 \eeq
Multiply (\ref{eq1}) by $w_1$ and
multiply (\ref{eq2}) by $\bar{z}_1$, then subtract to eliminate $u_1$:
\beq \label{eq3}
u_2(\bar{z}_2 w_1 - \bar{z}_1 w_2) + u_3 (\bar{z}_3 w_1 - \bar{z}_1 w_3) = 0
\eeq

Likewise, multiply (\ref{eq1}) by $w_2$ and multiply (\ref{eq2}) by
$\bar{z}_2$ to eliminate
$u_2$:

\beq \label{eq4}
u_1 (\bar{z}_1 w_2 - \bar{z}_2 w_1) + u_3 (\bar{z}_3 w_2 - \bar{z}_2 w_3) = 0
\eeq

This determines $(u_1,u_2, u_3)$ up to multiplication by a complex constant.
We find
\beq \label{eq5}
u_1 = -\frac{\bar{z}_3 w_2 - \bar{z}_2 w_3 }{\bar{z}_1 w_2 - \bar{z}_2 w_1} u_3\eeq
\beq
u_2 = -\frac{ \bar{z}_3 w_1 - \bar{z}_1 w_3      }{\bar{z}_2 w_1 - \bar{z}_1 w_2    } u_3
\eeq

The value of $u_3$ is determined (up to multiplication by a
complex constant  of unit norm $e^{i \theta} $) by the condition
\beq \label{norm}
|u_1|^2 + |u_2|^2 + |u_3|^2 = 1
\eeq
In fact it is uniquely determined by the additional
condition that
the determinant of the matrix $M$ with columns $z,w,u$ is $1$
(since we want this matrix to be in $SU(3)$).
\end{proof}

We now state our result:
\begin{theorem}
The intersection cohomology of the universal imploded cross-section of $SU(3)$ is
\begin{eqnarray*}
IH^j( (T^*SU(3))_{{\rm impl}} )&\cong&\left\{\begin{array}{ll} \mathbb{R} & j=0,4 \\ 0 & j=1,~2,~3,\mbox{and}~j\geq 5\end{array}\right.
\end{eqnarray*}
\end{theorem}

\begin{proof}
As explained earlier, we only need to compute the cohomology of $Y$.
Define $$ W = \{ (z,w) \in \CC^3 \times \CC^3: z \cdot w =0,
      |z|^2 + |w|^2 =  1, |z|^2 < \epsilon   ~\mbox{or}~ |w|^2 < \epsilon \} \subset Y.$$
So $W$ retracts to the disjoint union of two copies of $S^5$.


Then $Y = W \cup X$, and we compute the cohomology of $Y$ using the Mayer-Vietoris sequence for $W$ and $X$ as follows.

The intersection $ W \cap X$ retracts to the disjoint union of
$$ \{(z,w) \in \CC^3 \times \CC^3:  z \cdot w =0,~|z|^2 = \epsilon, |w|^2 =  1- \epsilon\} $$
and
$$ \{(z,w) \in \CC^3 \times \CC^3:  z \cdot w =0,~|z|^2 =1- \epsilon, |w|^2 =  \epsilon\}. $$
Each of these spaces is another copy of $SU(3) $.

So our Mayer-Vietoris sequence involves

$$ \cdots\to H^i(W \cup X) \to H^i(W) \oplus H^i(X) \to H^i (W \cap X) \to \cdots,$$
in other words
$$ \to H^i (Y) \to H^i (S^5) \oplus H^i (S^5) \oplus H^i (SU(3)) \to H^i (SU(3))\oplus H^i (SU(3)) \to .$$

The cohomology ring of  $SU(3)$ is the exterior algebra generated by
one generator $u$
of degree $3$ and another generator $v$ of degree $5$
$$ < u,v: v^2 = u^2 =0, uv= -vu>, $$
so we get $$H^0(SU(3)) \cong H^3 (SU(3)) \cong H^5(SU(3))  \cong H^8(SU(3)) \cong \RR$$
and all the others are $0$.

The Mayer-Vietoris sequence shows that the value of the map
$$\iota^* : H^3(W) \oplus H^3(X) \to H^3(W \cap X) $$ determines
the value of $H^3(W \cup X)$.
By definition, the map $\iota^*$ sends $(\alpha,\beta)\in
H^3(W) \oplus H^3(X)$  to $\beta|_{W\cap X}-\alpha|_{W\cap X}
\in H^3(W \cap X).$  Since $H^3(W)\cong 0$,  $\alpha$ is zero and
so $\iota^* : H^3(X) \to H^3(W \cap X)$  is just the
diagonal map
$x \mapsto (x,x)$ from $H^3(SU(3))$ to  $H^3(SU(3)) \oplus H^3(   SU(3)) $, which
is injective. From exactness of the sequence, it implies that
$H^3(W \cup X) = 0 $ and
$H^4(W \cup X) \cong \RR$.

As a result
$$ H^0(Y) \cong \RR \cong H^9(Y)$$
$$ H^1(Y) \cong 0 \cong H^8(Y)$$
$$ H^2(Y) \cong 0 \cong H^7(Y)$$
$$ H^3(Y) \cong 0 \cong H^6(Y)$$ 
$$ H^4(Y) \cong \RR \cong H^5(Y)$$ 
where we have used Poincar\'{e} duality since $Y$ is a compact smooth manifold.

So we conclude that
\begin{eqnarray*}
IH^j( (T^*SU(3))_{{\rm impl}} )&\cong& IH^j(U)\\&\cong &\left\{\begin{array}{ll} H^j(Y) & j\leq 4\\ 0 & j \geq 5\end{array}\right.\\  &\cong&\left\{\begin{array}{ll} \mathbb{R} & j=0,4 \\ 0 & j=1,~2,~3,\mbox{and}~j\geq 5\end{array}\right.
\end{eqnarray*}

\end{proof}

\end{document}